\documentclass[oneside,english]{amsart}

\usepackage[T1]{fontenc}
\usepackage[utf8]{inputenc}
\usepackage{array}
\usepackage{multirow}
\usepackage{amstext}
\usepackage{amsthm}
\usepackage{amsmath}
\usepackage{amssymb}
\usepackage{hyperref}
\usepackage{graphicx}
\usepackage{subcaption}
\usepackage[bb=boondox]{mathalfa}
\usepackage[table]{xcolor}

\usepackage[
backend=biber,
style=alphabetic,
sorting=ynt
]{biblatex}
\makeatletter

\numberwithin{equation}{section}
\numberwithin{figure}{section}
\theoremstyle{plain}
\newtheorem{thm}{\protect\theoremname}
  \theoremstyle{definition}
  \newtheorem{defn}[thm]{\protect\definitionname}
  \theoremstyle{remark}
  
  \theoremstyle{plain}
  \newtheorem{conjecture}[thm]{\protect\conjecturename}
  \newtheorem{lem}[thm]{\protect\lemmaname}
  \theoremstyle{definition}
  
  \theoremstyle{plain}
  \newtheorem{cor}[thm]{\protect\corollaryname}
  \providecommand{\corollaryname}{Corollary}

\makeatother

\usepackage{babel}
  \providecommand{\conjecturename}{Conjecture}
  \providecommand{\definitionname}{Definition}
  \providecommand{\examplename}{Example}
  \providecommand{\remarkname}{Remark}
\providecommand{\lemmaname}{Lemma}
\providecommand{\theoremname}{Theorem}

\begin{document}
\title[Approximation Functions to $li(x)$ with $\displaystyle O\left(\sqrt{\frac{x}{\log(x)}}\right)$ Error]{Defining Upper and Lower Bounding Functions of $li(x)$ with $\displaystyle O\left(\sqrt{\frac{x}{\log(x)}}\right)$ Error Using Truncated Asymptotic Series}

\author{Jonatan Gomez \\ jgomezpe@unal.edu.co \\ Computer Systems Engineering \\ Universidad Nacional de Colombia}

\begin{abstract}
    This paper introduces approximation functions of logarithmic integral $li(x)$ for all $x\ge e$: (1) $\displaystyle li_{\underline{\omega},\alpha}(x) = \frac{x}{\log(x)}\left( \alpha\frac{\underline{m}!}{\log^{\underline{m}}(x)} + \sum_{k=0}^{\underline{m}-1}\frac{k!}{\log^{k}(x)} \right)$, and (2) $\displaystyle li_{\overline{\omega},\beta}=\frac{x}{\log(x)}\left( \beta\frac{\overline{m}!}{\log^{\overline{m}}(x)} + \sum_{k=0}^{\overline{m}-1}\frac{k!}{\log^{k}(x)} \right)$
    with $0 < \omega < 1$ a real number, $\alpha \in \{ 0, \underline{\kappa}\log(x) \}$, $\underline{m} = \lfloor \underline{\kappa}\log(x) \rfloor$, $\beta \in \{ \overline{\kappa}\log(x), 1 \}$, $\overline{m} = \lfloor \overline{\kappa}\log(x) \rfloor$, and $\underline{\kappa} < \overline{\kappa}$ the solutions of $\kappa(1-\log(\kappa)) = \omega$. Since the error of approximating the $li(x)$ function using Stieltjes asymptotic approximation series $\displaystyle li_{*}(x) = \frac{x}{\log(x)}\sum_{k=0}^{n-1}\frac{k!}{\log^{k}(x)} + (\log(x)-n)\frac{xn!}{\log^{n+1}(x)}$, with $\displaystyle n = \lfloor \log(x) \rfloor$ for all $x\ge e$, satisfies $\displaystyle |\varepsilon(x)| = |li(x)-li_{*}(x)| \le 1.265692883422\ldots$, by using Stirling's approximation formula and some facts about $\log(x)$ and floor functions, this paper shows that $\displaystyle \varepsilon_{0}(x) = li(x) - li_{\underline{1/2},0}(x)$, $\displaystyle \underline{\varepsilon}(x) = li(x) - li_{\underline{1/2},\underline{\kappa}\log(x)}(x)$, $\displaystyle \overline{\varepsilon}(x) = li(x) - li_{\overline{1/2},\overline{\kappa}\log(x)}(x)$, and $\varepsilon_{1}(x) = li_{\overline{1/2},1}(x) - li(x)$ belong to $O\left(\sqrt{\frac{x}{\log(x)}}\right)$. Moreover, taking into account some theoretical and experimental results, this paper conjectures that $li_{0}(x) \le \pi(x) \le li_{1}(x)$ and $\underline{li}(x) \le \pi(x) \le \overline{li}(x)$ for all $x \ge e$, here $\pi(x)$ is the prime counting function. Finally, this paper shows that if one of those conjectures is true then the Riemann Hypothesis is true.
    
\end{abstract}
\maketitle

\section{\label{sec:introduction}Introduction}
    The prime number theorem, proved independently by Hadamard \cite{Hadamard1896} and de la Vall\'ee Poussin \cite{delaVallee1896}, describes the asymptotic behavior\footnote{function $f(x)$ has the same asymptotic behaviour of function $g(x)$ iff $\displaystyle \lim_{x\to\infty}\frac{f(x)}{g(x)}=1$} of the prime counting $\pi(x)$ as follows: 
    
    \begin{thm}
        \label{thm:PNT} $\displaystyle \pi(x)\sim \frac{x}{\log(x)}$, i.e., $\displaystyle \lim_{x \to \infty}\frac{\pi(x)}{\left(\frac{x}{\log(x)}\right)}=1$
    \end{thm}

    Dirichlet in \cite{LejeuneDirichlet1838} conjectured that a better approximation to $\pi(x)$ is the offset logarithmic integral function $\text{Li}(x)=li(x)-li(2)$ which captures the notion of prime numbers density function.

    \begin{defn}
    \label{{def:li}}
        $\displaystyle li(x)=\int_{0}^{x}\frac{dt}{\log(t)}$ is the \textbf{logarithmic integral} function. Here, $\log(x)$ is the natural logarithm function.
    \end{defn}

    Function $li(x)$ is defined for all positive real number $x \ne 1$ where function $\displaystyle \frac{1}{\log(t)}$ has a singularity.
    
    In \cite{delaVallee1899}, de la Vallée Poussin  proved that $\pi(x) \sim \text{Li}(x)$.

    \begin{thm}
        \label{thm:pi_Li}
        $\displaystyle \pi(x) = \text{Li}(x) + O(xe^{-a\sqrt{\log(x)}})$ as $x \to \infty$ for some positive constant $a$.
    \end{thm}

    In fact, assuming the Riemann Hypothesis it is possible to show that the absolute difference between $\pi(x)$ and $li(x)$ is $\left|\pi(x)-li(x)\right|= O\left(\sqrt{x}\log(x)\right)$, see \cite{narkiewicz2000}. Moreover, to prove the Riemann Hypothesis is equivalent to prove the following statement: $\left|\pi(x)-li(x)\right|= O\left(x^{1/2+\alpha}\right)$ for any $\alpha>0$.
    
    Function $li(x)$ has an asymptotic expansion in the form of a divergent power series that must be truncated after some finite number of terms for obtaining an approximated value of $li(x)$.
    
    \begin{lem}
        \label{lem:asympli}
        $\displaystyle li(x) \sim \frac{x}{\log(x)}\sum_{k=0}^{\infty}\frac{k!}{\log^{k}(x)}$.
    \end{lem} 
    
    Stieltjes in \cite{ASENS_1886_3_3__201_0} determined the optimal truncation point of the asymptotic approximation series (series is truncated at term minimizing $\varepsilon_n(x)$) is $n=\lfloor \log(x) \rfloor$. Also, van Boven et al. in \cite{Boven2012AsymptoticSO} introduced a fractional term that works like a 'remainder' of the asymptotic expansion series, see Definition \ref{def:li_*}.

    \begin{defn}
        \label{def:li_*}
        The \textbf{Stieltjes asymptotic approximation} of $li(x)$ is the function $\displaystyle li_{*}(x) = \frac{\alpha x n!}{\log^{n+1}(x)} + \frac{x}{\log(x)}\sum_{k=0}^{n-1}\frac{k!}{\log^{k}(x)}$, defined for all $x\ge e$, with  $\displaystyle n = \lfloor \log(x) \rfloor$ and $\alpha=\log(x)-n$.
    \end{defn}

    \begin{lem}
        \label{lem:Stieltjes}
        $\displaystyle \frac{x}{\log(x)}\sum_{k=0}^{n-1}\frac{k!}{\log^{k}(x)} \le li_{*}(x) \le \frac{x}{\log(x)}\sum_{k=0}^{n}\frac{k!}{\log^{k}(x)}$ for all $x\ge e$, with  $\displaystyle n = \lfloor \log(x) \rfloor$.
    \end{lem}
    \begin{proof}
        Follows from Definition \ref{def:li_*} and $0 \le \alpha=\log(x)-n \le 1$.
    \end{proof}

    The Stieltjes asymptotic approximation error is $\varepsilon(x)=li(x)-li_{*}(x)$ for all real number $x \ge 0$. In \cite{gomez2024stieltjes}, we proved that $\displaystyle |\varepsilon(x)| \le 1.265692883422\ldots$ for all $x \ge e$.

    In this paper, we bound the error of approximating $li(x)$ when using the asymptotic series if it is truncated at term $\displaystyle m=\lfloor \kappa \log(x) \rfloor$ with $0 < \kappa < e$. For doing this, we use the upper bound of $n!$ defined by Stirling's approximation formula\cite{Stirling}:
    
    \begin{equation}
        \label{eq:Stirling}
        \displaystyle \sqrt{2\pi n}\left(\frac{n}{e}\right)^ne^{\frac{1}{12n+1}}<n!<\sqrt{2\pi n}\left(\frac{n}{e}\right)^ne^{\frac{1}{12n}} \text{ for all } n\ge1.
    \end{equation}

\section{Main Results}

\subsection{Stirling's Fractional Extension}
    The following bounds can be obtained from Stirling's approximation formula (Equation \ref{eq:Stirling}):

    \begin{lem}
        \label{lem:StirlingKappa}
        $\displaystyle \frac{m!}{\log^{m+1}(x)} < \frac{C_{\kappa}}{x^{\kappa(1 - \log(\kappa))} \sqrt{\log(x)}}$ for all $0 < \kappa \le e$ and all $\displaystyle x \ge e^{\frac{1}{\kappa}}$ real numbers, with $\displaystyle m = \lfloor \kappa \log(x) \rfloor$ and $\displaystyle C_{\kappa} = \sqrt{\frac{2\pi}{\kappa}}e^{\frac{13}{12}}$.
    \end{lem}
    \begin{proof}
        Since $\displaystyle x \ge e^{\frac{1}{\kappa}}$ then $\displaystyle \kappa\log(x) \ge 1$ and $m = \lfloor \kappa \log(x) \rfloor \ge 1$. Now, $\displaystyle \frac{m!}{\log^{m+1}(x)}<\frac{\sqrt{2\pi m}\left(\frac{m}{e}\right)^me^{\frac{1}{12m}}}{\log^{m+1}(x)}$ (divide Stirling's approximation right inequality by $\log^{m+1}(x)$). Clearly, $\displaystyle \frac{m!}{\log^{m+1}(x)} < \frac{\sqrt{2\pi \kappa \log(x)} \left(\frac{\kappa\log(x)}{e}\right)^m e^{\frac{1}{12m}}}{\log^{m+1}(x)} = \frac{\sqrt{2\pi \kappa} \left(\frac{\kappa}{e}\right)^m e^{\frac{1}{12m}}}{\sqrt{\log(x)}}$ (since $\displaystyle m \le \kappa \log(x)$). So, $\displaystyle \frac{m!}{\log^{m+1}(x)} < \frac{\sqrt{2\pi \kappa} \left(\frac{e}{\kappa}\right)\left(\frac{\kappa}{e}\right)^{\kappa\log(x)} e^{\frac{1}{12m}}}{\sqrt{\log(x)}}$ (Lemma \ref{lem:floorceilfrac}), i.e., $\displaystyle \frac{m!}{\log^{m+1}(x)} < \frac{\sqrt{2\pi \kappa} \left(\frac{e}{\kappa}\right)e^{\frac{1}{12m}}}{x^{\kappa(1-\log(\kappa))}\sqrt{\log(x)}} = \frac{\sqrt{\frac{2\pi}{\kappa}}e^{1+\frac{1}{12m}}}{x^{\kappa(1 - \log(\kappa))} \sqrt{\log(x)}}$ (Corollary \ref{cor:alphax}). The inequality holds because $\displaystyle \sqrt{\frac{2\pi}{\kappa}}e^{1+\frac{1}{12m}} \le \sqrt{\frac{2\pi}{\kappa}}e^{1+\frac{1}{12}} = C_{\kappa}$ for all $m\ge1$. 
    \end{proof}

    \begin{cor}
        \label{cor:StirlingKappa}
        $\displaystyle \frac{xm!}{\log^{m+1}(x)} < C_{\kappa}\frac{x^{1 - \kappa(1 - \log(\kappa))}}{ \sqrt{\log(x)}}$ for all real numbers $0 < \kappa \le e$ and $\displaystyle x \ge e^{\frac{1}{\kappa}}$ with $\displaystyle m = \lfloor \kappa \log(x) \rfloor$ and $\displaystyle C_{\kappa} = \sqrt{\frac{2\pi}{\kappa}}e^{\frac{13}{12}}$.
    \end{cor}
    \begin{proof}
        Multiply each term in inequality of Lemma \ref{lem:StirlingKappa} by $x$.
    \end{proof}

    \begin{lem}
        \label{lem:StirlingKappaL}
        $\displaystyle  \frac{B_{\kappa}}{x^{\kappa(1 - \log(\kappa))} \sqrt{\log(x)}} < \frac{m!}{\log^{m+1}(x)} $ for all real numbers $0 < \kappa \le e$ and $\displaystyle x \ge e^{\frac{1}{\kappa}}$, with $\displaystyle m = \lfloor \kappa \log(x) \rfloor$ and $\displaystyle B_{\kappa} = \sqrt{\frac{\pi}{2\kappa}}$.
    \end{lem}
    \begin{proof}
        $\displaystyle \kappa\log(x) \ge 1$ and $m = \lfloor \kappa \log(x) \rfloor \ge 1$ ($\displaystyle x \ge e^{\frac{1}{\kappa}}$). Now, $\displaystyle \frac{m!}{\log^{m+1}(x)} > \frac{\sqrt{2\pi m}\left(\frac{m}{e}\right)^me^{\frac{1}{12m+1}}}{\log^{m+1}(x)}$ (divide Stirling's approximation left inequality by $\log^{m+1}(x)$). Clearly, $\displaystyle \frac{m!}{\log^{m+1}(x)} > \frac{\sqrt{2\pi } m^{m+1} \kappa \left(\frac{\kappa}{e}\right)^m e^{\frac{1}{12m+1}}}{\sqrt{m}\left(\kappa\log(x)\right)^{m+1}} > \frac{\sqrt{2\pi } m^{m+1} \kappa \left(\frac{\kappa}{e}\right)^m e^{\frac{1}{12m+1}}}{\sqrt{m}(m+1)^{m+1}} = \frac{\sqrt{2\pi } \kappa \left(\frac{\kappa}{e}\right)^m e^{\frac{1}{12m+1}}}{\sqrt{m}} \left(\frac{m}{m+1}\right)^{m+1} \ge \frac{\sqrt{2\pi \kappa} \left(\frac{\kappa}{e}\right)^m e^{\frac{1}{12m+1}-1}}{2\sqrt{\log(x)}}$ (since $\displaystyle m \le \kappa \log(x)$, $\displaystyle \left(\frac{m}{m+1}\right)^{m} \ge \frac{1}{e}$, and $\displaystyle \frac{m}{m+1} \ge \frac{1}{2}$). So $\displaystyle \frac{m!}{\log^{m+1}(x)} > \frac{\sqrt{2\pi \kappa} \left(\frac{e}{\kappa}\right)\left(\frac{\kappa}{e}\right)^{\kappa\log(x)} e^{\frac{1}{12m+1}-1}}{2\sqrt{\log(x)}}$ (Lemma \ref{lem:floorceilfrac}), i.e., $\displaystyle \frac{m!}{\log^{m+1}(x)} >  \frac{\sqrt{\frac{\pi}{2\kappa}}e^{\frac{1}{12m+1}}}{x^{\kappa(1 - \log(\kappa))} \sqrt{\log(x)}}$ (Corollary \ref{cor:alphax}). The inequality holds because $\displaystyle \sqrt{\frac{\pi}{2\kappa}}e^{\frac{1}{12m+1}} \ge \sqrt{\frac{\pi}{2\kappa}} = B_{\kappa}$ for all $m\ge1$. 
    \end{proof}

    \begin{cor}
        \label{cor:StirlingKappaL}
        $\displaystyle  B_{\kappa}\frac{x^{1 - \kappa(1 - \log(\kappa))}}{ \sqrt{\log(x)}} < \frac{xm!}{\log^{m+1}(x)}$ for all real numbers $0 < \kappa \le e$ and $\displaystyle x \ge e^{\frac{1}{\kappa}}$ with $\displaystyle m = \lfloor \kappa \log(x) \rfloor$ and $\displaystyle B_{\kappa} = \sqrt{\frac{\pi}{2\kappa}}$.
    \end{cor}
    \begin{proof}
        Multiply each term in inequality of Lemma \ref{lem:StirlingKappaL} by $x$.
    \end{proof}

    \begin{thm}
        \label{thm:StirlingKappa}
        $\displaystyle  \frac{xm!}{\log^{m+1}(x)} = \Theta \left( \frac{x^{1 - \kappa(1 - \log(\kappa))}}{ \sqrt{\log(x)}} \right)$ for all real numbers $0 < \kappa \le e$ and $\displaystyle x \ge e^{\frac{1}{\kappa}}$ with $\displaystyle m = \lfloor \kappa \log(x) \rfloor$.
    \end{thm}
    \begin{proof}
        Follows from Corollaries \ref{cor:StirlingKappa} and \ref{cor:StirlingKappaL}.
    \end{proof}

   We consider Theorem \ref{cor:StirlingKappa} for values of $\kappa$ that satisfy $\displaystyle \kappa(1-\log(\kappa)) = \omega$ for some real value $0 < \omega < 1$, see Figure \ref{fig:kappa(1-ln(kappa))}.

    \begin{figure}[htb]
        \centering
            \includegraphics[width=\linewidth,height=50mm, keepaspectratio]{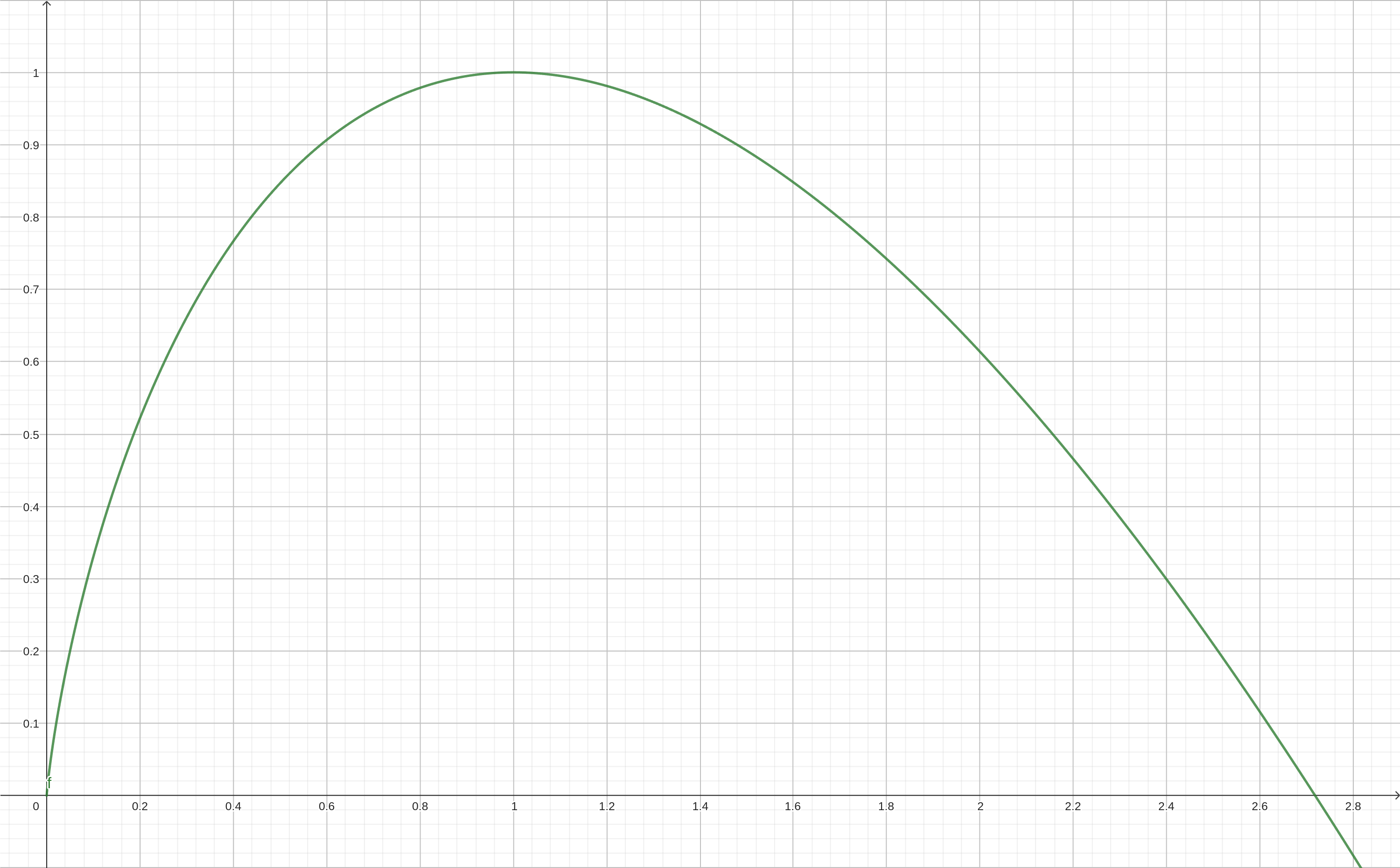}
        \caption{Behaviour of $\displaystyle \kappa(1-\log(\kappa))$ for $0 \le \kappa \le e$. }
        \label{fig:kappa(1-ln(kappa))}
    \end{figure}

\subsection{$\omega\alpha$-Asymptotic Approximations of $li_{*}(x)$}
    \begin{defn}
        \label{def:notation}
        Let $0 < \omega < 1$ and $x \ge e$ be real numbers.
        \begin{enumerate}
            \item $0 < \underline{\kappa} < 1$ is the real number such that $ \omega = \underline{\kappa}(1-\log(\underline{\kappa}))$,
            \item  $\underline{m} = \lfloor \underline{\kappa} \log(x) \rfloor$,
            \item  $\underline{\alpha} = \underline{\kappa} \log(x) - \underline{m}$,
            \item $1 < \overline{\kappa} < e$ is the real number such that $ \omega = \overline{\kappa}(1-\log(\overline{\kappa}))$,
            \item  $\overline{m} = \lfloor \overline{\kappa} \log(x) \rfloor$,
            \item  $\overline{\alpha} = \overline{\kappa} \log(x) - \overline{m}$,
        \end{enumerate}
    \end{defn}
    
    \begin{defn}
        \label{def:li_-}
        Let $0 < \omega < 1$ be a real number, $\alpha \in \{ 0, \underline{\alpha} \}$, and $\underline{m}$ and $\underline{\alpha}$ as defined in \ref{def:notation}. The $\underline{\omega}\alpha$-\textbf{asymptotic approximation} of $li(x)$ defined for all $x\ge e$ is the function $\displaystyle li_{\underline{\omega},\alpha}(x) = \frac{x}{\log(x)}\left( \alpha\frac{\underline{m}!}{\log^{\underline{m}}(x)} + \sum_{k=0}^{\underline{m}-1}\frac{k!}{\log^{k}(x)} \right)$.  
    \end{defn}

    \begin{lem}
        \label{lem:bounds_-}
        $\displaystyle li_{\underline{\omega},\alpha}(x) < li_{*}(x)$ for all $x \ge e$. 
    \end{lem}
    \begin{proof}
        We have two cases: (1) If $\displaystyle \lfloor \underline{\kappa} \log(x) \rfloor = \underline{m} = n = \lfloor \log(x) \rfloor$ then $\displaystyle li_{*}(x) - li_{\underline{\omega},\alpha}(x) = \frac{x}{\log(x)}\left[\left(\frac{(\log(x)-n) n!}{\log^{n}(x)} + \sum_{k=0}^{n-1}\frac{k!}{\log^{k}(x)} \right) - \left(\frac{\alpha n!}{\log^{n}(x)} + \sum_{k=0}^{n-1}\frac{k!}{\log^{k}(x)} \right)\right] = \frac{x n! \left(\log(x)-n-\alpha\right)}{\log^{n+1}(x)} \ge \frac{x n! \left(\log(x)-n-(\underline{\kappa}\log(x)-\underline{m})\right)}{\log^{n+1}(x)} = \frac{x n! \left(1-\underline{\kappa}\right)\log(x)}{\log^{n+1}(x)} > 0$ (since $\alpha \le \underline{\alpha}= \underline{\kappa}\log(x)$). (2) If $\displaystyle \underline{m} < n$ then  $\displaystyle li_{\underline{\omega},\alpha}(x) < \frac{x}{\log(x)} \sum_{k=0}^{\underline{m}}\frac{k!}{\log^{k}(x)}$ (since $\alpha < 1$). Therefore, $\displaystyle li_{\underline{\omega},\alpha}(x) < \frac{x}{\log(x)} \sum_{k=0}^{n-1}\frac{k!}{\log^{k}(x)}$ (since $\underline{m}<n$), i.e., $\displaystyle li_{\underline{\omega},\alpha}(x) < li_{*}(x)$ (Lemma \ref{lem:Stieltjes}).
    \end{proof}

    \begin{defn}
        \label{def:li_+}
        Let $0 < \omega < 1$ be a real number, $\alpha \in \{ \overline{\alpha}, 1 \}$, and $\overline{m}$ and $\overline{\alpha}$ as defined in \ref{def:notation}. The $\overline{\omega}\alpha$-\textbf{asymptotic approximation} of $li(x)$ defined for all $x\ge e$ is the function $\displaystyle li_{\overline{\omega},\alpha}=\frac{x}{\log(x)}\left( \alpha\frac{\overline{m}!}{\log^{\overline{m}}(x)} + \sum_{k=0}^{\overline{m}-1}\frac{k!}{\log^{k}(x)} \right)$. 
    \end{defn}  

    \begin{lem}
        \label{lem:bounds_+}
        $\displaystyle li_{*}(x) \le li_{\overline{\omega},\alpha}(x)$ for all $x\ge e$.    
    \end{lem}
    \begin{proof}
        We have two cases: (1) If $\displaystyle \lfloor \overline{\kappa} \log(x) \rfloor = \overline{m} = n = \lfloor \log(x) \rfloor$ then $\displaystyle li_{\overline{\omega},\alpha}(x) - li_{*}(x) = \frac{x}{\log(x)}\left[\left(\frac{\alpha n!}{\log^{n}(x)} + \sum_{k=0}^{n-1}\frac{k!}{\log^{k}(x)} \right) - \left(\frac{(\log(x)-n) n!}{\log^{n}(x)} + \sum_{k=0}^{n-1}\frac{k!}{\log^{k}(x)} \right)\right] = \frac{x n! \left(\alpha - \log(x) + n\right)}{\log^{n+1}(x)} \ge \frac{x n! \left((\overline{\kappa}\log(x)-\overline{m}) - \log(x) + n \right)}{\log^{n+1}(x)} = \frac{x n! \log(x)\left(\overline{\kappa}-1\right)}{\log^{n+1}(x)} > 0$ (since $\alpha \ge \overline{\alpha}= \overline{\kappa}\log(x)$). (2) If $\displaystyle \overline{m} > n$ then  $\displaystyle li_{\overline{\omega},\alpha}(x) \ge \frac{x}{\log(x)} \sum_{k=0}^{\overline{m}-1}\frac{k!}{\log^{k}(x)}$ (removes fraction in $\displaystyle li_{\overline{\omega},\alpha}(x)$), so $\displaystyle li_{\overline{\omega},\alpha}(x) \ge \frac{x}{\log(x)} \sum_{k=0}^{n}\frac{k!}{\log^{k}(x)}$ (since $\overline{m}>n$), i.e., $\displaystyle li_{\overline{\omega},\alpha}(x) \ge li_{*}(x)$ (Lemma \ref{lem:Stieltjes}).
    \end{proof}
    
    Bounds for the error of approximating $li_{*}(x)$ using the $\omega\alpha$-asymptotic approximations can be obtained as follows.
    
    \begin{lem}
        \label{lem:g_varepsilon_-}
        $\displaystyle 0 \le \varepsilon^{*}_{\underline{\omega},\alpha}(x) = li_{*}(x) - li_{\underline{\omega},\alpha}(x) \le \left(\frac{C_{\underline{\kappa}}x^{1-\omega}}{\sqrt{\log(x)}}\right)\sum_{k=1}^{n}\frac{1}{\log^{k}(x)}\prod_{i=1}^{k} (\underline{m}+i)$ for all real number $\displaystyle x \ge e$, with $n=\lfloor \log(x) \rfloor$, $\displaystyle C_{\underline{\kappa}} = \sqrt{\frac{2\pi}{\underline{\kappa}}}e^{\frac{13}{12}}$ and $\displaystyle \underline{m}$ and $\underline{\kappa}$ as defined in \ref{def:notation}.
    \end{lem} 
    \begin{proof}
        $\displaystyle 0 \le \varepsilon^{*}_{\underline{\omega},\alpha}$  (Lemma \ref{lem:bounds_-}). Now, $\displaystyle li_{*}(x) \le \frac{x}{\log(x)}\sum_{k=0}^{n}\frac{k!}{\log^{k}(x)}$ with $n=\lfloor \log(x) \rfloor$ (Lemma \ref{lem:Stieltjes}) and $\displaystyle li_{\underline{w}, \alpha}(x) \ge \frac{x}{\log(x)}\sum_{k=0}^{\underline{m}-1}\frac{k!}{\log^{k}(x)}$, with $\underline{m} = \lfloor \underline{\kappa} \log(x) \rfloor$ (Remove term $\displaystyle \alpha\frac{\underline{m}!}{\log^{\underline{m}}(x)}$). So, $\displaystyle \varepsilon^{*}_{\underline{\omega},\alpha} \le \frac{x}{\log(x)}\sum_{k=\underline{m}}^{n}\frac{k!}{\log^{k}(x)} = \frac{x\underline{m}!}{\log^{\underline{m}+1}(x)}\sum_{k=1}^{n}\frac{1}{\log^{k}(x)}\prod_{i=1}^{k} (\underline{m}+i) \le \left(\frac{C_{\underline{\kappa}}x^{1-\omega}}{\sqrt{\log(x)}}\right)\sum_{k=1}^{n}\frac{1}{\log^{k}(x)}\prod_{i=1}^{k} (\underline{m}+i)$ (Corollary \ref{cor:StirlingKappa}).
    \end{proof}

    \begin{cor}
        \label{cor:s_varepsilon_-}
        $\displaystyle 0 \le \varepsilon^{*}_{\underline{\omega},\alpha}(x) = li_{*}(x) - li_{\underline{\omega},\alpha}(x) \le S_{\underline{\omega}}x^{1-\omega}\sqrt{\log(x)}$ for all real number $\displaystyle x \ge e$ with $\displaystyle S_{\underline{\omega}} = (1-\underline{\kappa})C_{\underline{\kappa}} = \sqrt{\frac{2\pi}{\underline{\kappa}}}(1-\underline{\kappa})e^{\frac{13}{12}}$ and $\underline{\kappa}$ as defined in \ref{def:notation}.
    \end{cor} 
    \begin{proof}
        $\displaystyle 0 \le \varepsilon^{*}_{\underline{\omega},\alpha} \le \left(\frac{C_{\underline{\kappa}}x^{1-\omega}}{\sqrt{\log(x)}}\right)\sum_{k=1}^{n}\frac{1}{\log^{k}(x)}\prod_{i=1}^{k} (\underline{m}+i)$ with $n = \lfloor \log(x) \rfloor$ (Lemma \ref{lem:g_varepsilon_-}). $\displaystyle \varepsilon^{*}_{\underline{\omega},\alpha} \le \frac{C_{\underline{\kappa}}x^{1-\omega}(n-\underline{m})}{\sqrt{\log(x)}}$ (Lemma \ref{lem:s_sumprod_n-m_1}). $\displaystyle \varepsilon^{*}_{\underline{\omega},\alpha} \le \frac{C_{\underline{\kappa}}x^{1-\omega}(\log(x) - \underline{\kappa}\log(x))}{\sqrt{\log(x)}} \le C_{\underline{\kappa}}(1-\underline{\kappa})x^{1-\omega}\sqrt{\log(x)} = S_{\underline{\omega}}x^{1-\omega}\sqrt{\log(x)}$.
    \end{proof}

    \begin{cor}
        \label{cor:varepsilon_-}
        $\displaystyle 0 \le \varepsilon^{*}_{\underline{\omega},\alpha}(x) = li_{*}(x) - li_{\underline{\omega},\alpha}(x) \le D_{\underline{\omega}}\frac{x^{1-\omega}}{\sqrt{\log(x)}}$ for all real number $\displaystyle x \ge e$ with $\displaystyle D_{\underline{\omega}} = \sqrt{\frac{8\pi}{\underline{\kappa}}}e^{\frac{13}{12}}\left(\frac{1+\underline{\kappa}}{1-\underline{\kappa}}\right)$ and $\underline{\kappa}$ as defined in \ref{def:notation}.
    \end{cor} 
    \begin{proof}
        $\displaystyle 0 \le \varepsilon^{*}_{\underline{\omega},\alpha} \le \left(\frac{C_{\underline{\kappa}}x^{1-\omega}}{\sqrt{\log(x)}}\right)\sum_{k=1}^{n}\frac{1}{\log^{k}(x)}\prod_{i=1}^{k} (\underline{m}+i)$  (Lemma \ref{lem:g_varepsilon_-}). So, $\displaystyle \varepsilon^{*}_{\underline{\omega},\alpha} \le \left(\frac{C_{\underline{\kappa}}x^{1-\omega}}{\sqrt{\log(x)}}\right)2\left(\frac{\log(x)+\underline{m}}{\log(x)-\underline{m}}\right)$ (Lemma \ref{lem:sumprod_n-m_1}). Then, $\displaystyle \varepsilon^{*}_{\underline{\omega},\alpha} \le \left(\frac{2C_{\underline{\kappa}}x^{1-\omega}}{\sqrt{\log(x)}}\right)\left(\frac{1+\underline{\kappa}}{1-\underline{\kappa}}\right) = D_{\underline{\omega}}\frac{x^{1-\omega}}{\sqrt{\log(x)}}$.
    \end{proof}

    \begin{lem}
        \label{lem:g_varepsilon_+}
        $\displaystyle 0 \le  \varepsilon^{*}_{\overline{\omega},\alpha}(x) \le \left(\frac{C_{\overline{\kappa}}x^{1-\omega}}{\sqrt{\log(x)}}\right)\left(1+\sum_{k=1}^{\overline{m}-n}\log^{k}(x)\prod_{i=0}^{k-1}\frac{1}{\overline{m}-i}\right)$ for all real number $\displaystyle x \ge e$  with = $\varepsilon^{*}_{\overline{\omega},\alpha}(x) = li_{\overline{\omega},\alpha}(x) - li_{*}(x)$, $n=\lfloor \log(x) \rfloor$, $\displaystyle C_{\overline{\kappa}} = \sqrt{\frac{2\pi}{\overline{\kappa}}}e^{\frac{13}{12}}$, and $\displaystyle \overline{m}$ and $\overline{\kappa}$ as defined in \ref{def:notation}.
    \end{lem} 
    \begin{proof}
        $\displaystyle 0 \le \varepsilon^{*}_{\overline{\omega}}$  (Lemma \ref{lem:bounds_+}). Now, $\displaystyle li_{*}(x) \ge \frac{x}{\log(x)}\sum_{k=0}^{n-1}\frac{k!}{\log^{k}(x)}$ with $n=\lfloor \log(x) \rfloor$ (Lemma \ref{lem:Stieltjes}) and $\displaystyle li_{\overline{w},\alpha}(x) \le \frac{x}{\log(x)}\sum_{k=0}^{\overline{m}}\frac{k!}{\log^{k}(x)}$, (since $\displaystyle 0 \le \alpha < 1$ in $li_{\overline{\omega}}(x)$). So, $\displaystyle \varepsilon^{*}_{\overline{\omega},\alpha} \le \frac{x}{\log(x)}\sum_{k=n}^{\overline{m}}\frac{k!}{\log^{k}(x)} = \frac{x \overline{m}!}{\log^{\overline{m}+1}(x)}\left(1+\sum_{k=1}^{\overline{m}-n}\log^{k}(x)\prod_{i=0}^{k-1}\frac{1}{\overline{m}-i}\right)$. Then, $\displaystyle \varepsilon^{*}_{\overline{\omega},\alpha} \le \left(\frac{C_{\overline{\kappa}}x^{1-\omega}}{\sqrt{\log(x)}}\right)\left(1+\sum_{k=1}^{\overline{m}-n}\log^{k}(x)\prod_{i=0}^{k-1}\frac{1}{\overline{m}-i}\right)$ (Corollary \ref{cor:StirlingKappa}).
    \end{proof} 

    \begin{cor}
        \label{cor:s_varepsilon_+}
        $\displaystyle 0 \le \varepsilon^{*}_{\overline{\omega}}(x) = li_{\overline{\omega}}(x) - li_{*}(x) \le S_{\overline{\omega}}x^{1-\omega}\sqrt{\log(x)}$ for all real number $\displaystyle x \ge e$ with $\displaystyle S_{\overline{\omega}} = \sqrt{\frac{2\pi}{\overline{\kappa}}}\left(\overline{\kappa}+1\right)e^{\frac{13}{12}}$ and $\overline{\kappa}$ as defined in \ref{def:notation}.
    \end{cor} 
    \begin{proof}
        $\displaystyle 0 \le \varepsilon^{*}_{\overline{\omega},\alpha}(x) \le \left(\frac{C_{\overline{\kappa}}x^{1-\omega}}{\sqrt{\log(x)}}\right)\left(1+\sum_{k=1}^{\overline{m}-n}\log^{k}(x)\prod_{i=0}^{k-1}\frac{1}{\overline{m}-i}\right)$ (Lemma \ref{lem:g_varepsilon_+}). Now, $\displaystyle \varepsilon^{*}_{\overline{\omega},\alpha} \le \left(\frac{C_{\overline{\kappa}}x^{1-\omega}}{\sqrt{\log(x)}}\right)\left(1+\overline{m}-n\right) \le \left(\frac{C_{\overline{\kappa}}x^{1-\omega}}{\sqrt{\log(x)}}\right) \left(2+\overline{m}-\log(x)\right) $ (Lemma \ref{lem:s_sumprod_n-m_2} and $\log(x)-1 < n $). Then, $\displaystyle \varepsilon^{*}_{\overline{\omega},\alpha} \le \left(\frac{C_{\overline{\kappa}}x^{1-\omega}}{\sqrt{\log(x)}}\right)\left(\overline{\kappa}+1\right)\log(x) = S_{\overline{\omega}}x^{1-\omega}\sqrt{\log(x)}$ (since $\displaystyle 2+\overline{m}-\log(x) \le 2 + \overline{\kappa}log(x) - \log(x) \le (2 + \overline{\kappa} - 1)\log(x)$).
    \end{proof}    

    \begin{cor}
        \label{cor:varepsilon_+}
        $\displaystyle 0 \le \varepsilon^{*}_{\overline{\omega},\alpha}(x) = li_{\overline{\omega},\alpha}(x) - li_{*}(x) \le D_{\overline{\omega}}\frac{x^{1-\omega}}{\sqrt{\log(x)}}$ for all real number $\displaystyle x \ge e$ with $\displaystyle D_{\overline{\omega}} = \left(\frac{\overline{\kappa}+3}{\overline{\kappa}-1}\right)\sqrt{\frac{2\pi}{\overline{\kappa}}}e^{\frac{13}{12}}$ and $\overline{\kappa}$ as defined in \ref{def:notation}..
    \end{cor} 
    \begin{proof}
        $\displaystyle 0 \le \varepsilon^{*}_{\overline{\omega},\alpha}(x) \le \left(\frac{C_{\overline{\kappa}}x^{1-\omega}}{\sqrt{\log(x)}}\right)\left(1+\sum_{k=1}^{\overline{m}-n}\log^{k}(x)\prod_{i=0}^{k-1}\frac{1}{\overline{m}-i}\right)$ (Lemma \ref{lem:g_varepsilon_+}). Now, $\displaystyle \varepsilon^{*}_{\overline{\omega},\alpha} \le \left(\frac{C_{\overline{\kappa}}x^{1-\omega}}{\sqrt{\log(x)}}\right)\left(1+2\left(\frac{2\log(x)}{\overline{m}-\log(x)+1}\right)\right) \le \left(\frac{C_{\overline{\kappa}}x^{1-\omega}}{\sqrt{\log(x)}}\right) \left(1+\frac{4}{\overline{\kappa}-1}\right)$ (By Lemma \ref{lem:sumprod_n-m_2}, $n \le \log(x)$, and $\displaystyle \overline{m} + 1 \ge \overline{\kappa}\log(x)$). Therefore, $\displaystyle \varepsilon^{*}_{\overline{\omega},\alpha} \le D_{\overline{\omega}}\frac{x^{1-\omega}}{\sqrt{\log(x)}}$.
    \end{proof}    

\subsection{$\displaystyle \frac{1}{2}$-Asymptotic Approximation of $li_{*}(x)$}
    
    We are interested in the asymptotic approximation functions defined by $\displaystyle \omega=\frac{1}{2}$, i.e., $\displaystyle li_{0}(x) = li_{\underline{1/2},0}(x)$, $\displaystyle \underline{li}(x) = li_{\underline{1/2},\underline{\alpha}}(x)$, $\displaystyle \overline{li}(x) = li_{\overline{1/2},\overline{\alpha}}(x)$, and $\displaystyle li_{1}(x) = li_{\overline{1/2},1}(x)$. Then, $\underline{\kappa} \approx 0.18668231$, $\overline{\kappa} \approx 2.155535203$, $\displaystyle \underline{D} = D_{\underline{1/2}} \approx 50.01822612$, and $\overline{D} = D_{\overline{1/2}}  \approx 22.50549859$. 

    \begin{lem}
        \label{lem:varepsilon_1/2-}
        $\displaystyle \underline{\varepsilon}^{*}(x) = li_{*}(x) - \underline{li}(x) \le \varepsilon_{0}^{*}(x) = li_{*}(x) - li_{0}(x) \le \underline{D}\sqrt{\frac{x}{\log(x)}}$ for all real number $\displaystyle x \ge e$, with $\displaystyle \underline{D} \approx 50.0182261266$.
    \end{lem} 
    \begin{proof}
        Follows from Corollary \ref{cor:varepsilon_-}, $\displaystyle 1-\omega=\frac{1}{2}$, $\underline{\kappa} \approx 0.18668231$, and $li_{0}(x) \le \underline{li}(x)$ (since $0 \le \underline{\alpha})$.
    \end{proof}

    \begin{lem}
        \label{lem:varepsilon_1/2+}
        $\displaystyle \overline{\varepsilon}^{*}(x) =  \overline{li}(x) - li_{*}(x) \le \varepsilon_{1}^{*}(x) = li_{1}(x) - li_{*}(x) \le \overline{D}\sqrt{\frac{x}{\log(x)}}$ for all real number $\displaystyle x \ge e$, with $\displaystyle \overline{D} \approx 22.5054985892$.
    \end{lem} 
    \begin{proof}
        Follows from Corollary \ref{cor:varepsilon_-}, $\displaystyle 1-\omega=\frac{1}{2}$, $\overline{\kappa} \approx 2.155535203$, and $li_{1}(x) \ge \overline{li}(x)$ (since $\overline{\alpha}) \le 1$.
    \end{proof}    

    Tables \ref{table:li1-li2-1} and \ref{table:li1-li2-2} show values of $\displaystyle \varepsilon^{*}_{0}(x) =  li_{*}(x) - li_{0}(x)$, $\displaystyle \frac{\varepsilon^{*}_{0}(x)}{\sqrt{\displaystyle \frac{x}{\log(x)}}}$, $\displaystyle \varepsilon^{*}_{1}(x) =  li{1}(x) - li_{*}(x)$, $\displaystyle \frac{\varepsilon^{*}_{1}(x)}{\sqrt{\displaystyle \frac{x}{\log(x)}}}$, $\displaystyle \underline{\varepsilon}^{*}(x) =  li_{*}(x) - \underline{li}(x)$, $\displaystyle \frac{\underline{\varepsilon^{*}}(x)}{\sqrt{\displaystyle \frac{x}{\log(x)}}}$, $\displaystyle \overline{\varepsilon}^{*}(x) =  \overline{li}(x) - li_{*}(x)$, and $\displaystyle \frac{\overline{\varepsilon}^{*}(x)}{\sqrt{\displaystyle \frac{x}{\log(x)}}}$ for $x=10^1,\ldots,10^{29}$. Notice that computed ratios (multiplicative constant) are smaller than our estimations $\displaystyle \underline{D} \approx 50.0182261266$, and $\overline{D} \approx 22.50549859$.

    \begin{table}[htbp]
        \centering
        \begin{tabular}{|r|r|c|r|c|}
            \hline
            $x$ & $\displaystyle \varepsilon^{*}_{0}(x)$ & $\displaystyle \frac{\varepsilon^{*}_{0}(x)}{\sqrt{\displaystyle \frac{x}{\log(x)}}}$ & $\displaystyle \varepsilon^{*}_{1}(x)$ & $\displaystyle \frac{\varepsilon^{*}_{1}(x)}{\sqrt{\displaystyle \frac{x}{\log(x)}}}$ \\
            \hline
$10^{1}$ & $6.20$ & $2.97$ & $7.50$ & $3.60$ \\
$10^{2}$ & $30.10$ & $6.45$ & $19.04$ & $4.08$ \\
$10^{3}$ & $33.23$ & $2.76$ & $48.41$ & $4.02$ \\
$10^{4}$ & $160.26$ & $4.86$ & $131.96$ & $4.00$ \\
$10^{5}$ & $189.66$ & $2.03$ & $371.69$ & $3.98$ \\
$10^{6}$ & $1006.37$ & $3.74$ & $1081.41$ & $4.01$ \\
$10^{7}$ & $1228.85$ & $1.56$ & $3214.43$ & $4.08$ \\
$10^{8}$ & $6824.92$ & $2.92$ & $9688.79$ & $4.15$ \\
$10^{9}$ & $41026.02$ & $5.90$ & $29513.14$ & $4.24$ \\
$10^{10}$ & $48256.77$ & $2.31$ & $90615.57$ & $4.34$ \\
$10^{11}$ & $290952.54$ & $4.63$ & $279938.15$ & $4.45$ \\
$10^{12}$ & $349620.27$ & $1.83$ & $869086.04$ & $4.56$ \\
$10^{13}$ & $2111307.71$ & $3.65$ & $2709093.04$ & $4.68$ \\
$10^{14}$ & $2574265.41$ & $1.46$ & $8473443.11$ & $4.81$ \\
$10^{15}$ & $15560289.93$ & $2.89$ & $26579732.02$ & $4.93$ \\
$10^{16}$ & $97333587.55$ & $5.90$ & $83583833.49$ & $5.07$ \\
$10^{17}$ & $115969563.37$ & $2.29$ & $263411915.59$ & $5.21$ \\
$10^{18}$ & $722847752.94$ & $4.65$ & $831717944.55$ & $5.35$ \\
$10^{19}$ & $871653553.89$ & $1.82$ & $2630570690.35$ & $5.50$ \\
$10^{20}$ & $5417789519.97$ & $3.67$ & $8332555187.06$ & $5.65$ \\
$10^{21}$ & $6595055877.28$ & $1.45$ & $26429805276.11$ & $5.81$ \\
$10^{22}$ & $40899906101.43$ & $2.91$ & $83932453654.78$ & $5.97$ \\
$10^{23}$ & $259252743272.20$ & $5.96$ & $266850519584.53$ & $6.14$ \\
$10^{24}$ & $310499878990.33$ & $2.30$ & $849248495359.50$ & $6.31$ \\
$10^{25}$ & $1961049546483.68$ & $4.70$ & $2705224738143.37$ & $6.49$ \\
$10^{26}$ & $2368374639649.50$ & $1.83$ & $8624643594697.38$ & $6.67$ \\
$10^{27}$ & $14911710286211.34$ & $3.71$ & $27518163013281.19$ & $6.86$ \\
$10^{28}$ & $18135450483667.70$ & $1.45$ & $41019570013812.60$ & $3.29$ \\
$10^{29}$ & $113880130793454.00$ & $2.94$ & $130998204833437.00$ & $3.38$ \\
        \hline
        \end{tabular}
        
        \caption{Values of $\displaystyle \varepsilon^{*}_{0}(x) =  li_{*}(x) - li_{0}(x)$, $\displaystyle \frac{\varepsilon^{*}_{0}(x)}{\sqrt{\displaystyle \frac{x}{\log(x)}}}$, $\displaystyle \varepsilon^{*}_{1}(x) =  li_{1}(x) - li_{*}(x)$, and $\displaystyle \frac{\varepsilon^{*}_{1}(x)}{\sqrt{\displaystyle \frac{x}{\log(x)}}}$ for $x=10^1,\ldots,10^{29}$. Results are obtained with the Python program inside the Colab notebook (lifrac.ipynb) freely available at \cite{GomezPrimesGit}.}
        \label{table:li1-li2-1} 
    \end{table}
  
    \begin{table}[htbp]
        \centering
        \begin{tabular}{|r|r|c|r|c|}
            \hline
            $x$ & $\displaystyle \underline{\varepsilon}^{*}(x)$ & $\displaystyle \frac{\underline{\varepsilon}^{*}(x)}{\sqrt{\displaystyle \frac{x}{\log(x)}}}$ & $\displaystyle \overline{\varepsilon}^{*}(x)$ & $\displaystyle \frac{\overline{\varepsilon}^{*}(x)}{\sqrt{\displaystyle \frac{x}{\log(x)}}}$ \\
            \hline
$10^{1}$ & $4.33$ & $2.07$ & $7.37$ & $3.53$ \\
$10^{2}$ & $11.43$ & $2.45$ & $18.42$ & $3.95$ \\
$10^{3}$ & $27.16$ & $2.25$ & $45.95$ & $3.81$ \\
$10^{4}$ & $75.45$ & $2.29$ & $122.71$ & $3.72$ \\
$10^{5}$ & $170.10$ & $1.82$ & $338.06$ & $3.62$ \\
$10^{6}$ & $567.14$ & $2.10$ & $961.65$ & $3.57$ \\
$10^{7}$ & $1220.89$ & $1.55$ & $2793.79$ & $3.54$ \\
$10^{8}$ & $4538.22$ & $1.94$ & $8226.36$ & $3.53$ \\
$10^{9}$ & $12766.04$ & $1.83$ & $24468.26$ & $3.52$ \\
$10^{10}$ & $37187.86$ & $1.78$ & $73319.47$ & $3.51$ \\
$10^{11}$ & $123256.82$ & $1.96$ & $220934.61$ & $3.51$ \\
$10^{12}$ & $306955.57$ & $1.61$ & $668631.16$ & $3.51$ \\
$10^{13}$ & $1130329.16$ & $1.95$ & $2030435.84$ & $3.51$ \\
$10^{14}$ & $2538585.40$ & $1.44$ & $6182563.84$ & $3.51$ \\
$10^{15}$ & $10061753.29$ & $1.86$ & $18866284.73$ & $3.50$ \\
$10^{16}$ & $28738270.47$ & $1.74$ & $57669986.66$ & $3.50$ \\
$10^{17}$ & $87854839.94$ & $1.73$ & $176522092.15$ & $3.49$ \\
$10^{18}$ & $296082176.27$ & $1.90$ & $540876548.13$ & $3.48$ \\
$10^{19}$ & $756838941.68$ & $1.58$ & $1658551646.91$ & $3.46$ \\
$10^{20}$ & $2833677349.55$ & $1.92$ & $5088463711.51$ & $3.45$ \\
$10^{21}$ & $6455418097.84$ & $1.41$ & $15616220827.15$ & $3.43$ \\
$10^{22}$ & $26004576057.05$ & $1.85$ & $47928112273.66$ & $3.41$ \\
$10^{23}$ & $73877196562.03$ & $1.70$ & $147095889815.31$ & $3.38$ \\
$10^{24}$ & $232271215052.04$ & $1.72$ & $451306226910.06$ & $3.35$ \\
$10^{25}$ & $783545598025.84$ & $1.87$ & $1384015627524.11$ & $3.32$ \\
$10^{26}$ & $2036657012715.75$ & $1.57$ & $4241595234936.98$ & $3.28$ \\
$10^{27}$ & $7656067872952.33$ & $1.90$ & $12988339862084.01$ & $3.23$ \\
$10^{28}$ & $17619025700877.80$ & $1.41$ & $40421839917705.90$ & $3.24$ \\
$10^{29}$ & $71372189120343.00$ & $1.84$ & $126539989037161.00$ & $3.26$ \\
        \hline
        \end{tabular}
        
        \caption{Values of $\displaystyle \underline{\varepsilon}^{*}(x) =  li_{*}(x) - \underline{li}(x)$, $\displaystyle \frac{\underline{\varepsilon^{*}}(x)}{\sqrt{\displaystyle \frac{x}{\log(x)}}}$, $\displaystyle \overline{\varepsilon}^{*}(x) =  \overline{li}(x) - li_{*}(x)$, and $\displaystyle \frac{\overline{\varepsilon}^{*}(x)}{\sqrt{\displaystyle \frac{x}{\log(x)}}}$ for $x=10^1,\ldots,10^{29}$. Results are obtained with the Python program inside the Colab notebook (lifrac.ipynb) freely available at \cite{GomezPrimesGit}.}
        \label{table:li1-li2-2} 
    \end{table}

\subsection{$\displaystyle \frac{1}{2}$-Approximation Functions, $li(x)$, and $\pi(x)$}
    Finally, we study functions $li(x)$ and $\pi(x)$ in terms of functions $li_{-}(x)$ and $li_{+}(x)$.

    \begin{thm}
        \label{thm:varepsilon_li_1/2-li}
        $\displaystyle li(x) - f(x) \le li_{0}(x) \le \underline{li}(x)  \le li(x) \le \overline{li}(x) \le li_{1}(x) \le li(x) + f(x)$ for all real number $\displaystyle x \ge e$, with $f(x)=\underline{D}\sqrt{\frac{x}{\log(x)}} - 1.265692883422\ldots$ and $\underline{D} \approx 50.0182261266$.
    \end{thm} 
    \begin{proof}
        Follows from Lemmas \ref{lem:varepsilon_1/2-} and \ref{lem:varepsilon_1/2+}, $\overline{D} < \underline{D}$, and  $\displaystyle |\varepsilon(x)| \le 1.265692883422\ldots$ for all $x \ge e$ as shown by Gomez in \cite{gomez2024stieltjes}.
    \end{proof} 

    \begin{lem}
        \label{lem:li0pi}
        $\displaystyle li_{0}(x) \le \pi(x)$ for all $ e \le x < e^{\frac{4}{\underline{\kappa}}} \approx 2.09*10^{9}$  with $\displaystyle \underline{\kappa} \approx 0.18668231$ as defined in \ref{def:notation}.
    \end{lem}
    \begin{proof}
        $\displaystyle li_{0}(x)=0$  and $\pi(x) \ge 1$ for all $\displaystyle e \le x < e^{\frac{1}{\underline{\kappa}}} \approx 212.02471377$. Also, $\displaystyle li_{0}(x)=\frac{x}{\log(x)}$ for all $\displaystyle e^{\frac{1}{\underline{\kappa}}} \le x < e^{\frac{2}{\underline{\kappa}}} \approx 44954.4792$ and $\displaystyle \frac{x}{\log(x)} < \pi(x)$ for all $x\ge 17$ (see \cite{10.1215/ijm/1255631807}, page 117 in \cite{narkiewicz2000} and Corollary 5.2, Equation (5.2) in \cite{Dusart2018}). Now, $\displaystyle li_{0}(x)=\frac{x}{\log(x)}\left(1+\frac{1}{\log(x)}\right)$ for all $\displaystyle e^{\frac{2}{\underline{\kappa}}} \le x < e^{\frac{3}{\underline{\kappa}}} \approx 9531460.5955$ and $\displaystyle \frac{x}{\log(x)}\left(1+\frac{1}{\log(x)}\right) < \pi(x)$ for all $x \ge 599$ (see Corollary 5.2, Equation (5.3) in \cite{Dusart2018}). Finally, $\displaystyle li(x)=\frac{x}{\log(x)}\left(1+\frac{1}{\log(x)}+\frac{2}{\log^{2}(x)}\right)$ for all $\displaystyle e^{\frac{3}{\underline{\kappa}}} \le x < e^{\frac{4}{\underline{\kappa}}} \approx 2.09*10^{9}$ and $\displaystyle \frac{x}{\log(x)}\left(1+\frac{1}{\log(x)}+\frac{2}{\log^{2}(x)}\right) < \pi(x)$ for all $x \ge 88789$ (see Corollary 5.2, Equation (5.4) in \cite{Dusart2018}). 
    \end{proof}

    \begin{lem}
        \label{lem:li_pi}
        $\displaystyle \underline{li}(x) \le \pi(x)$ for all $ e \le x < e^{\frac{3}{\underline{\kappa}}} \approx 9531460.5955$  with $\displaystyle \underline{\kappa} \approx 0.18668231$.
    \end{lem}
    \begin{proof}
        $\displaystyle \underline{li}(x)=\underline{\kappa}x$ for all $\displaystyle e \le x < e^{\frac{1}{\underline{\kappa}}} \approx 212.02471$. Then, $\displaystyle \underline{li}(x) \le \underline{li}(5) \approx 0.93341 < 1 = \pi(2) \le \pi(x)$ for all $e \le x \le 5$. Moreover, $\displaystyle \underline{li}(x) \le \underline{li}(16) \approx 2.98692 < 3 = \pi(5) \le \pi(x)$ for all $5 \le x \le 16$. Also, $\displaystyle \underline{li}(x) \le \underline{li}(32) \approx 5.97383 < 6 = \pi(16) \le \pi(x)$ for all $16 \le x \le 32$. Now, $\displaystyle \underline{li}(x) < f(x) = \frac{x}{\log(x)}$ for all $\displaystyle e \le x < e^{\frac{1}{\underline{\kappa}}} \approx 212.02471$ and $\displaystyle f(x) < \pi(x)$ for all $x \ge 17$ (see \cite{10.1215/ijm/1255631807} and Corollary 5.2, Equation (5.2) in \cite{Dusart2018}). Next, $\displaystyle \underline{li}(x)=\frac{x}{\log(x)}\left(1+\frac{\underline{\kappa}\log(x)-1}{\log(x)}\right)$ for all $\displaystyle e^{\frac{1}{\underline{\kappa}}} \le x < e^{\frac{2}{\underline{\kappa}}} \approx 44954.4792$. Then, $\displaystyle \underline{li}(x) \le \frac{x}{\log(x)} \left(1+\frac{\underline{\kappa}\log(599)-1}{\log(x)}\right) \approx \frac{x}{\log(x)} \left(1+\frac{0.19388}{\log(x)}\right)$ for all $\displaystyle e^{\frac{1}{\underline{\kappa}}} \le x < 599$ and $\displaystyle \frac{x}{\log(x)}\left(1+\frac{0.5}{\log(x)}\right) < \pi(x)$ for all $x \ge 59$ (see page 13 in \cite{Dusart1998}). Moreover, $\displaystyle \underline{li}(x) \le f(x) = \frac{x}{\log(x)}\left(1+\frac{1}{\log(x)}\right)$ for all $\displaystyle e^{\frac{1}{\underline{\kappa}}} \le x < e^{\frac{2}{\underline{\kappa}}}$ and $\displaystyle f(x) < \pi(x)$ for all $x \ge 599$ (see Corollary 5.2, Equation (5.3) in \cite{Dusart2018}). Finally, $\displaystyle li_{0}(x)=\frac{x}{\log(x)}\left(1+\frac{1}{\log(x)}+\frac{2(\underline{\kappa}\log(x)-2)}{\log^{2}(x)}\right)$ for all $\displaystyle e^{\frac{2}{\underline{\kappa}}} \le x < e^{\frac{3}{\underline{\kappa}}}$. Then, $\displaystyle \underline{li}(x) \le \frac{x}{\log(x)} \left(1+\frac{1}{\log(x)}+\frac{2(\underline{\kappa}\log(88789)-2)}{\log^{2}(x)}\right) \approx \frac{x}{\log(x)} \left(1+\frac{1}{\log(x)}+\frac{0.25412}{\log^{2}(x)}\right)$ for all $\displaystyle e^{\frac{2}{\underline{\kappa}}} \le x < 88789$ and we have that $\displaystyle \frac{x}{\log(x)}\left(1+\frac{1}{\log(x)}+\frac{1.8}{\log^{2}(x)}\right) < \pi(x)$ for all $x \ge 32299$ (Theorem 1.10 part  (6) in \cite{Dusart1998}). Moreover, we have that $\displaystyle \underline{li}(x) \le f(x) = \frac{x}{\log(x)}\left(1+\frac{1}{\log(x)}+\frac{2}{\log^{2}(x)}\right)$ for all $\displaystyle e^{\frac{2}{\underline{\kappa}}} \le x < e^{\frac{3}{\underline{\kappa}}}$ and $\displaystyle f(x) < \pi(x)$ for all $x \ge 88789$ (see Corollary 5.2, Equation (5.4) in \cite{Dusart2018}). 
    \end{proof}

    First lines in proof of Lemma \ref{lem:li_pi} consider intervals $\displaystyle [x_i, x_{i+1}]$ such that $\displaystyle x_{i+1}=\max_{n>x_i}\{\underline{li}(n) \le \pi(x_i)\}$ with $x_0=2$. This process can be generalized as follows:

    \begin{lem}
        \label{lem:intervals}
        Let $X \ge e$ be a real number.  If there exists $I \in \mathbb{N}$ such that $\displaystyle x_{i+1}=\max_{n>x_i}\{\underline{li}(n) \le \pi(x_i)\}$ for all $i=0,\ldots, I$ and $X \le x_{I+1}$ with $x_{0} = 2$ then $\displaystyle \underline{li}(x) \le \pi(x)$ for all $e \le x \le X$.
    \end{lem}
    \begin{proof}
        If $e \le x \le X$ then there is $i=0,\ldots, I$ such that $x_{i} < x \le x_{i+1}$. Clearly, $\underline{li}(x) \le \underline{li}(x_{i+1}) \le \pi(x_{i}) \le \pi(x)$. 
    \end{proof}

    \begin{cor}
        \label{cor:li_pi-program}
        $\displaystyle \underline{li}(x) \le \pi(x)$ for all $ e \le x < 2.09*10^{9}$.
    \end{cor}
    \begin{proof}
        We use the C++ program inside the Colab note (lifrac-c++.ipynb) freely available at \cite{GomezPrimesGit} which uses Lemma \ref{lem:intervals}. We obtain $I=13408$ and $x_{I+1} = 2090132958$.
    \end{proof}
    
    \begin{lem}
        \label{lem:li2pi}
        $\displaystyle  \pi(x) \le \overline{li}(x) \le li_{1}(x)$ for all $e \le x \le 1.39*10^{17}$.
    \end{lem}
    \begin{proof}
        Follows from $li(x) \le \overline{li}(x) \le li_{1}(x)$ and $\pi(x) < li(x)$ for all $2 < x \le 1.39*10^{17}$ (see Corollary 1 in \cite{Platt2016}).
    \end{proof}
    
    Table \ref{table:li1/2-pi} compares values of the difference between $\displaystyle \frac{1}{2}$-asymptotic approximation functions and function $\displaystyle \pi(x)$ for $n=10^2,\ldots,10^{29}$. Notice that function $\pi(x)$ is bounded by $\underline{li}(x)$ and $\overline{li}(x)$ for the reported values of $x$. This behavior along results in Lemmas \ref{lem:li0pi}, \ref{lem:li_pi}, and \ref{lem:li2pi}, and Corollary \ref{cor:li_pi-program} motivate us to conjecture that:

    \begin{conjecture}
        \label{conj:li1/2-pi}
        $\displaystyle li_{0}(x) \le \pi(x) \le li_{1}(x)$ for all $x \ge e$.
    \end{conjecture}

    \begin{conjecture}
        \label{conj:li1/2_pi-}
        $\displaystyle \underline{li}(x) \le \pi(x) \le \overline{li}(x)$ for all $x \ge e$.
    \end{conjecture}

    \begin{table}[htbp]
        \centering
        \begin{tabular}{|r|r|r|}
            \hline
            $x$ & $\pi(x)-\underline{li}(x)$ & $\overline{li}(x)-\pi(x)$ \\
            \hline
$10^{1}$ & $2.13$ & $9.57$ \\
$10^{2}$ & $6.33$ & $23.52$ \\
$10^{3}$ & $17.16$ & $55.95$ \\
$10^{4}$ & $58.45$ & $139.71$ \\
$10^{5}$ & $132.10$ & $376.06$ \\
$10^{6}$ & $437.14$ & $1091.65$ \\
$10^{7}$ & $881.89$ & $3132.79$ \\
$10^{8}$ & $3784.22$ & $8980.36$ \\
$10^{9}$ & $11065.04$ & $26169.26$ \\
$10^{10}$ & $34083.86$ & $76423.47$ \\
$10^{11}$ & $111668.82$ & $232522.61$ \\
$10^{12}$ & $268692.57$ & $706894.16$ \\
$10^{13}$ & $1021358.16$ & $2139406.84$ \\
$10^{14}$ & $2223695.40$ & $6497453.84$ \\
$10^{15}$ & $9009134.29$ & $19918903.73$ \\
$10^{16}$ & $25523638.47$ & $60884618.66$ \\
$10^{17}$ & $79898250.94$ & $184478681.15$ \\
$10^{18}$ & $274132621.27$ & $562826103.13$ \\
$10^{19}$ & $656961166.68$ & $1758429421.91$ \\
$10^{20}$ & $2610932705.55$ & $5311208355.51$ \\
$10^{21}$ & $5858023843.84$ & $16213615081.15$ \\
$10^{22}$ & $24070220849.05$ & $49862467481.66$ \\
$10^{23}$ & $66627010346.03$ & $154346076031.31$ \\
$10^{24}$ & $215124307774.04$ & $468453134188.06$ \\
$10^{25}$ & $728384617086.84$ & $1439176608463.11$ \\
$10^{26}$ & $1880765334594.75$ & $4397486913057.98$ \\
$10^{27}$ & $7147401214946.33$ & $13497006520090.01$ \\
$10^{28}$ & $16191280040503.80$ & $41849585578079.90$ \\
$10^{29}$ & $66820995497879.00$ & $131091182659625.00$ \\
            \hline
        \end{tabular}
        
        \caption{Values of the difference between $\displaystyle \frac{1}{2}$-asymptotic approximation functions and function $\displaystyle \pi(x)$ for $n=10^2,\ldots,10^{29}$. Results are obtained with the Python program inside the Colab notebook (lifrac.ipynb) freely available at \cite{GomezPrimesGit}.}
        \label{table:li1/2-pi} 
    \end{table}

    \begin{lem}
        \label{lem:varepsilon-li-li}
        If Conjecture \ref{conj:li1/2_pi-} holds then Conjecture \ref{conj:li1/2-pi} holds.
    \end{lem}
    \begin{proof}
        Follows from $li_{0}(x) \le \underline{li}(x)$ and $\overline{li}(x) \le li_{1}(x)$.
    \end{proof}  
    
    \begin{lem}
        \label{lem:varepsilon-li-pi}
        If Conjecture \ref{conj:li1/2-pi} holds then $\displaystyle |li(x)-\pi(x)| = O\left(\sqrt{\frac{x}{\log(x)}}\right)$ for all $x \ge e$.
    \end{lem}
    \begin{proof}
        $\displaystyle  |li(x)-\pi(x)| \le \underline{D} \sqrt{\frac{x}{\log(x)}} + 1.265692883422\ldots$ with $\underline{D} \approx 50.0182261266$ (Follows from Theorem \ref{thm:varepsilon_li_1/2-li} and Conjecture \ref{conj:li1/2-pi}). 
    \end{proof}

    \begin{thm}
        \label{thm:varepsilon-li-pi}
        If Conjecture \ref{conj:li1/2-pi} holds then Riemann's Conjecture holds.
    \end{thm}
    \begin{proof}
        Follows from Lemma \ref{lem:varepsilon-li-pi} and $\displaystyle f(x) = O\left(\sqrt{\frac{x}{\log(x)}}\right) $ then $f(x) = O(\sqrt{x})$.
    \end{proof}
  
\section{Auxiliary Facts and Proofs}
    We introduce some customary statements that we use for proving our main results.

    \begin{lem}
        \label{lem:alphax}
        $\displaystyle \alpha^{\alpha\log(x)} = x^{\alpha\log(\alpha)}$ for all $x > 0$ and all $\alpha > 0$. 
    \end{lem}
    \begin{proof}
        $\log(\alpha^{\alpha\log(x)})=\alpha\log(x)\log(\alpha)=\alpha\log(\alpha)\log(x)=\log\left(x^{\alpha\log(\alpha)})\right)$ (logarithm properties). Therefore, $\displaystyle \alpha^{\alpha\log(x)} = x^{\alpha\log(\alpha)}$.
    \end{proof}
    \begin{cor}
        \label{cor:alphax}
        $\displaystyle \left(\frac{\alpha}{e}\right)^{\alpha\log(x)}=\frac{1}{x^{\alpha(1-\log(\alpha))}}$ for all $x > 0$ and all $\alpha > 0$.
    \end{cor}
    \begin{proof}
        $\displaystyle \left(\frac{\alpha}{e}\right)^{\alpha\log(x)} = \frac{\alpha^{\alpha\log(x)}}{e^{\alpha\log(x)}} = \frac{x^{\alpha\log(\alpha)}}{e^{\alpha\log(x)}}$ (Lemma \ref{lem:alphax}). Therefore, $\displaystyle \left(\frac{\alpha}{e}\right)^{\alpha\log(x)} = \frac{x^{\alpha\log(\alpha)}}{x^{\alpha}} = \frac{1}{x^{\alpha(1-\log(\alpha))}}$.
    \end{proof}

    \begin{lem}
        \label{lem:floorceil}
        $a - 1 < \lfloor a \rfloor \le \lceil a \rceil < a + 1$ for all real number $0 < a$.
    \end{lem}
    \begin{proof} 
        Clearly, $a - 1 < \lfloor a \rfloor = a = \lceil a \rceil < a + 1$ if $a$ is a natural number. Now, $\lceil a \rceil - \lfloor a \rfloor = 1$ if $a$ is not a natural number. Then, $a < \lceil a \rceil$ so $a - \lceil a \rceil + \lfloor a \rfloor < \lfloor a \rfloor$, i.e., $a-1 \le \lfloor a \rfloor$. Moreover, $\lfloor a \rfloor < a$ then $\lceil a \rceil \le a + \lceil a \rceil - \lfloor a \rfloor$, i.e., $\lceil a \rceil \le a + 1$.
    \end{proof}

    \begin{lem}
        \label{lem:floorceilfrac}
        $\displaystyle \left(\frac{r}{s}\right)^{\lfloor a \rfloor} \le \left(\frac{s}{r}\right)\left(\frac{r}{s}\right)^{a}$ and $\displaystyle \left(\frac{r}{s}\right)\left(\frac{r}{s}\right)^{a} \le \left(\frac{r}{s}\right)^{\lceil a \rceil}$ for all pair of real numbers $0 < r \le s$ and all real number $0 < a$.
    \end{lem}
    \begin{proof} 
        Clearly, $\displaystyle \frac{r}{s} \le 1$ (since $0 < r \le s$). Since $a - 1 \le \lfloor a \rfloor$ (Lemma \ref{lem:floorceil}) then $\displaystyle \left(\frac{r}{s}\right)^{\lfloor a \rfloor} \le \left(\frac{r}{s}\right)^{a-1}=\left(\frac{s}{r}\right)\left(\frac{r}{s}\right)^{a}$. Now $a + 1 \ge \lceil a \rceil$ (Lemma \ref{lem:floorceil}) then $\displaystyle \left(\frac{r}{s}\right)^{\lceil a \rceil} \ge \left(\frac{r}{s}\right)^{a+1}=\left(\frac{r}{s}\right)\left(\frac{r}{s}\right)^{a}$.
    \end{proof}

    \begin{lem}
        \label{lem:sumk!/n^k}
        $\displaystyle \sum_{k=1}^{n+1} \frac{k!}{(n+1)^k} \le \sum_{k=1}^{n} \frac{k!}{n^k}$ for all $n \ge 1$.
    \end{lem}
    \begin{proof}
       We have that $\displaystyle \sum_{k=1}^{n} \frac{k!}{n^k} - \sum_{k=1}^{n+1} \frac{k!}{(n+1)^k} = \sum_{k=1}^{n} k! \left( \frac{1}{n^k} - \frac{1}{(n+1)^k} \right) - \frac{n!}{(n+1)^{n}} = \displaystyle \sum_{k=1}^{n-1} k! \left( \frac{1}{n^k} - \frac{1}{(n+1)^k} \right) + n! \left( \frac{1}{n^n} - \frac{2}{(n+1)^n} \right) \ge \frac{n!}{(n+1)^n} \left( \left(\frac{n+1}{n}\right)^n - 2\right) \ge 0$  (since $\displaystyle \left(\frac{n+1}{n}\right)^n$ is an increasing function then $\displaystyle \left(\frac{n+1}{n}\right)^n \ge \left(\frac{1+1}{1}\right)^1 = 2$).
    \end{proof}

    \begin{cor}
        \label{cor:sumk!/n^k}
        $\displaystyle \sum_{k=1}^{n} \frac{k!}{n^k} \le 1$ for all $n \ge 1$.
    \end{cor}
    \begin{proof}
        We use induction. For $m=1$ we have $\displaystyle \sum_{k=1}^{1} \frac{k!}{1^k} = 1$. For $m=n+1$ we have $\displaystyle \sum_{k=1}^{m} \frac{k!}{m^k} = \sum_{k=1}^{n+1} \frac{k!}{(n+1)^k} \le \sum_{k=1}^{n} \frac{k!}{n^k} \le 1$ (Lemma \ref{lem:sumk!/n^k} and induction step).
    \end{proof}
    
    \begin{lem}
        \label{lem:prod_n-m_1}
        $\displaystyle \frac{1}{y^{k}}\prod_{i=1}^{k}(m+i) \le \frac{1}{y^{j}}\prod_{i=1}^{j}(m+i) \le 1$ for all real number $y \ge 1$ and all natural numbers $m$, $k$, and $j$ such that $\displaystyle m < \lfloor y \rfloor$ and all natural number $1 \le j \le k \le \lfloor y \rfloor - m$.
    \end{lem} 
    \begin{proof}
        If $j=k$ the statement holds trivially. If $j<k$ then $\displaystyle \frac{\displaystyle \frac{1}{y^{k}}\prod_{i=1}^{k}(m+i)}{\displaystyle \frac{1}{y^j}\prod_{i=1}^{j}(m+i)} = \frac{1}{y^{k-j}}\prod_{i=1}^{k-j}(m+j+i) \le 1$ (since $m+j+i \le m+k = \lfloor y \rfloor$). Clearly, $\displaystyle \frac{1}{y^{k}}\prod_{i=1}^{k}(m+i) \le \frac{1}{y^1}\prod_{i=1}^{1}(m+i) = \frac{m+1}{y} \le \frac{\lfloor y \rfloor}{y} \le \frac{y}{y} = 1$.
    \end{proof}

    \begin{lem}
        \label{lem:s_sumprod_n-m_1}
        $\displaystyle  \sum_{k=1}^{n - m}\frac{1}{y^{k}}\prod_{i=1}^{k}(m+i) \le  n - m$ for all real number $y \ge 1$ and all natural number $m$ such that $m<n$, with $n=\lfloor y \rfloor$.
    \end{lem} 
    \begin{proof}
        $\displaystyle  \sum_{k=1}^{n - m}\frac{1}{y^{k}}\prod_{i=1}^{k}(m+i) \le  \sum_{k=1}^{n - m} 1 = n - m$ (Lemma \ref{lem:prod_n-m_1}).
    \end{proof}

    \begin{lem}
        \label{lem:sumprod_n-m_1}
        $\displaystyle  \sum_{k=1}^{n - m}\frac{1}{y^{k}}\prod_{i=1}^{k}(m+i) \le  2\left(\frac{y+m}{y-m}\right)$ for all real number $y \ge 1$ and all natural number $\displaystyle m < n$, with $n=\lfloor y \rfloor$.
    \end{lem} 
    \begin{proof}
       $\displaystyle \sum_{k=1}^{n-m}\frac{1}{y^{k}}\prod_{i=1}^{k}(m+i) = \sum_{k=1}^{m'} \frac{1}{y^{k}}\prod_{i=1}^{k}(m+i) + \sum_{k=m'+1}^{n-m} \frac{1}{y^{k}}\prod_{i=1}^{k}(m+i)$ for $\displaystyle m'=\left \lfloor \frac{n-m}{2} \right \rfloor$. Since every term in the first summation of the right part is greater or equal than every term in the second summation of the right part (Lemma \ref{lem:prod_n-m_1}), and the number of terms in the first summation is greater or equal than the number of terms in the second summation then $\displaystyle \sum_{k=1}^{n-m} \frac{1}{y^{k}}\prod_{i=1}^{k}\left(m+i\right) \le  2\sum_{k=1}^{m'} \frac{1}{y^{k}}\prod_{i=1}^{k}\left(m+i\right)$. Notice that, $\displaystyle \sum_{k=1}^{n-m} \frac{1}{y^{k}}\prod_{i=1}^{k}\left(m+i\right) \le 2\sum_{k=1}^{m'}\left ( \frac{m+k}{y} \right )^{k} \le 2\sum_{k=1}^{m'} \left( \frac{m+\left \lfloor \frac{n-m}{2} \right \rfloor}{y} \right )^{k} \le 2\sum_{k=1}^{m'} \left( \frac{n+m}{2y} \right )^{k} \le 2\sum_{k=1}^{m'} \left( \frac{y+m}{2y} \right )^{k} = 2\left(\sum_{k=0}^{m'} \left( \frac{y+m}{2y} \right )^{k} - 1 \right)$. Since $m <n = \lfloor y \rfloor \le y$ then $\displaystyle y+m < 2y$, $\displaystyle \frac{y+m}{2y} < 1$, and $\displaystyle \sum_{k=0}^{m'} \left ( \frac{y+m}{2y} \right )^{k} < \sum_{k=0}^{\infty} \left ( \frac{y+m}{2y} \right )^{k} \le \frac{1}{1-\displaystyle \frac{y+m}{2y}}=\frac{2y}{y-m}$ (geometric series). $\displaystyle \sum_{k=1}^{n-m} \frac{1}{y^{k}}\prod_{i=1}^{k}\left(m+i\right) \le 2 \left(\frac{2y}{y-m}-1\right) = 2\left(\frac{y+m}{y-m}\right)$.
    \end{proof}

    \begin{lem}
        \label{lem:prod_n-m_2}
        $\displaystyle y^{k}\prod_{i=0}^{k-1}\frac{1}{m-i} \le y^{j}\prod_{i=0}^{j-1}\frac{1}{m-i} \le \frac{y}{m} \le 1$ for all real number $y \ge 1$ and all natural numbers $m$, $k$, and $j$ such that $\displaystyle n < m$ and $1 \le j \le k \le m-n$. Here, $n=\lfloor y \rfloor$.
    \end{lem} 
    \begin{proof}
        If $k=j$ the statement holds trivially. If $j<k$ then $\displaystyle \frac{\displaystyle y^{k}\prod_{i=0}^{k-1}\frac{1}{m-i}}{\displaystyle y^{j}\prod_{i=0}^{j-1}\frac{1}{m-i}} = y^{k-j}\prod_{i=0}^{k-j-1}\frac{1}{m-j-i} \le y^{k-j}\prod_{i=0}^{k-j-1}\frac{1}{m-k+1} \le \left(\frac{y}{n+1}\right)^{k-j} \le \frac{y}{y}= 1$. Clearly, $\displaystyle y^{k}\prod_{i=0}^{k-1}\frac{1}{m-i} \le y\prod_{i=0}^{0}\frac{1}{m-i} = \frac{y}{m} \le 1$.
    \end{proof}

    \begin{lem}
        \label{lem:s_sumprod_n-m_2}
        $\displaystyle \sum_{k=1}^{m-n}y^{k}\prod_{i=0}^{k-1}\frac{1}{m-i} \le m - n$ for all real number $y \ge 1$ and all natural number $m$ such that $m-n \ge 1$, with $n=\lfloor y \rfloor$.
    \end{lem} 
    \begin{proof}
        Clearly, $\displaystyle k-1 \le m-n-1 = m - \lfloor y \rfloor -1 \le m - y$ and $m-i \ge m-k+1 \ge n+1$ for all $k=1,\ldots,m-n$ and all $i=0,\ldots,k-1$. Then, $\displaystyle \sum_{k=1}^{m-n}y^{k}\prod_{i=0}^{k-1}\frac{1}{m-i} \le \sum_{k=1}^{m-n}\left(\frac{y}{m-k+1}\right)^{k} \le \sum_{k=1}^{m-n}\left(\frac{y}{n+1}\right)^{k} \le \sum_{k=1}^{m-n}\left(\frac{y}{y}\right)^{k} = \sum_{k=1}^{m-n} 1 =  m - n$.
    \end{proof}
    
    \begin{lem}
        \label{lem:sumprod_n-m_2}
        $\displaystyle \sum_{k=1}^{m-n}y^{k}\prod_{i=0}^{k-1}\frac{1}{m-i} \le 2\left(\frac{2y}{m-y+1}\right)$ for all real number $y \ge 1$ and all natural number $m$ such that $m-n \ge 1$, with $n=\lfloor y \rfloor$.
    \end{lem} 
    \begin{proof}
        $\displaystyle \sum_{k=1}^{m-n} y^{k}\prod_{i=0}^{k-1}\frac{1}{m-i} = \sum_{k=1}^{m'} y^{k}\prod_{i=0}^{k-1}\frac{1}{m-i} + \sum_{k=m'+1}^{m-n} y^{k}\prod_{i=0}^{k-1}\frac{1}{m-i}$ for $\displaystyle m'=\left \lfloor \frac{m-n}{2} \right \rfloor$. Since every term in the first summation of the right part is greater or equal than every term in the second summation of the right part (Lemma \ref{lem:prod_n-m_2}) and the number of terms in the first summation is greater or equal than the number of terms in the second summation then $\displaystyle \sum_{k=1}^{m-n} y^{k}\prod_{i=0}^{k-1}\frac{1}{m-i} \le 2\sum_{k=1}^{m'}  y^{k}\prod_{i=0}^{k-1}\frac{1}{m-i}$. Since $m-i < m-k+1$ for all $i=0,\ldots, k-1$ then, $\displaystyle \sum_{k=1}^{m-n} y^{k}\prod_{i=0}^{k-1}\frac{1}{m-i} \le 2 \sum_{k=1}^{m'} \left ( \frac{y}{m-k+1} \right )^{k} \le 2\sum_{k=1}^{m'} \left( \frac{y}{m-m'+1} \right )^{k} = 2\sum_{k=1}^{m'} \left( \frac{y}{m-\left \lfloor \displaystyle \frac{m-n}{2} \right\rfloor + 1} \right )^{k} \le 2\sum_{k=1}^{m'} \left( \frac{y}{m - \displaystyle \frac{m-n}{2} + 1} \right )^{k} \le 2\sum_{k=1}^{m'} \left( \frac{2y}{n+m+2} \right )^{k} = 2\left(\sum_{k=0}^{m'} \left( \frac{2y}{n+m+2} \right )^{k} - 1\right)$. Since $n+m+2 \ge n+n+3 > 2(n+1) > 2\lceil y \rceil \ge 2y $ then $\displaystyle \frac{2y}{n+m+2} < 1$ and $\displaystyle \sum_{k=0}^{m'} \left( \frac{2y}{n+m+2} \right )^{k} < \sum_{k=0}^{\infty} \left( \frac{2y}{n+m+2} \right )^{k} = \frac{1}{1-\displaystyle \frac{2y}{n+m+2}} = \frac{n+m+2}{n+m+2-2y}$. So, $\displaystyle \sum_{k=1}^{m-n}y^{k}\prod_{i=0}^{k-1}\frac{1}{m-i} \le 2\left(\frac{n+m+2}{n+m+2-2y} - 1\right) = 2\left(\frac{2y}{n+m+2-2y}\right)$ (geometric series). Since $n+1 \ge y$ then $\displaystyle \sum_{k=1}^{m-n}y^{k}\prod_{i=0}^{k-1}\frac{1}{m-i} \le 2\left(\frac{2y}{m-y+1}\right)$.
    \end{proof}    

\section{Conclusions and Future Work}
    We establish bounds for logarithmic integral function $li(x)$ using truncated asymptotic approximation series: (1) $\displaystyle li_{\underline{\omega},\alpha}(x) = \frac{x}{\log(x)}\left( \alpha\frac{\underline{m}!}{\log^{\underline{m}}(x)} + \sum_{k=0}^{\underline{m}-1}\frac{k!}{\log^{k}(x)} \right)$, and (2) $\displaystyle li_{\overline{\omega},\beta}=\frac{x}{\log(x)}\left( \beta\frac{\overline{m}!}{\log^{\overline{m}}(x)} + \sum_{k=0}^{\overline{m}-1}\frac{k!}{\log^{k}(x)} \right)$
    for all $x\ge e$ with $0 < \omega < 1$, $\alpha \in \{ 0, \underline{\kappa}\log(x) \}$, $\underline{m} = \lfloor \underline{\kappa}\log(x) \rfloor$, $\beta \in \{ \overline{\kappa}\log(x), 1 \}$, $\overline{m} = \lfloor \overline{\kappa}\log(x) \rfloor$, and $\underline{\kappa} < \overline{\kappa}$ the solutions of $\kappa(1-\log(\kappa)) = \omega$. Such bounds are of the order $\displaystyle |li(x)-li_{0}(x)| = |li_{1}(x)-li(x)| = O\left(\sqrt{\frac{x}{\log(x)}}\right)$ for all $x \ge e$. We conjecture that $li_{0}(x) \le \pi(x) \le li_{1}(x)$ and $\underline{li}(x) \le \pi(x) \le \overline{li}(x)$ and show that if one of such conjectures is true, the Riemann Hypothesis is true. Our future work will concentrate on trying to demonstrate such conjectures.

\printbibliography

\end{document}